\documentclass{article}

 \usepackage[dvips]{graphics,color}
 \usepackage[centertags]{amsmath}
 \usepackage{amsfonts}
 \usepackage{amssymb}
 \usepackage{amsthm}

 \newtheorem{lemma}{Lemma}
 \newtheorem{theorem}{Theorem}
 
 \newtheorem{remark}{Remark}

 \newtheorem{example}{Example}
 \newtheorem{definition}{\bf Definition}

\begin{document}
 \begin{center}
 {\bf \Large Riemann surface of complex logarithm and multiplicative calculus }\\[0.2cm]
 {\bf Agamirza E. Bashirov}\\[0.2cm]
 {Department of Mathematics, Eastern Mediterranean University\\
 Gazimagusa, Mersin 10, Turkey and\\
 Institute of Control Systems, ANAS, Baku, Azerbaijan}\\
 \verb"agamirza.bashirov@emu.edu.tr"\\[0.2cm]
 {\bf Sajedeh Norozpour}\\[0.2cm]
 {Department of Mathematics, Eastern Mediterranean University\\
 Gazimagusa, Mersin 10, Turkey\\
 \verb"sajedeh.norozpour@cc.emu.edu.tr"} 
 \end{center}
{\small {\bf Abstract.} Could elementary complex analysis, which covers the topics such as algebra of complex numbers, elementary complex functions, complex differentiation and integration, series expansions of complex functions, residues and singularities, and introduction to conformal mappings, be made more elementary? In this paper we demonstrate that a little reorientation of existing elementary complex analysis brings a lot of benefits, including operating with single-valued logarithmic and power functions, making the Cauchy integral formula as a part of fundamental theorem of calculus, removal of residues and singularities, etc. Implicitly, this reorientation consists of resolving the multivalued nature of complex logarithm by considering its Riemann surface. But instead of the advanced mathematical concepts such as manifolds, differential forms, integration on manifolds, etc, which are necessary for introducing complex analysis in Riemann surfaces, we use rather elementary methods of multiplicative calculus. We think that such a reoriented elementary complex analysis could be especially successful as a first course in complex analysis for the students of engineering, physics, even applied mathematics programs who indeed do not see a second and more advanced course in complex analysis. It would be beneficial for the students of pure mathematics programs as well because it is more appropriate introduction to complex analysis on Riemann surfaces rather than the existing one.
}\\[0.3cm]
{\small {\bf Corresponding Author: Agamirza E. Bashirov} \verb"agamirza.bashirov@emu.edu.tr"; Tel: +90 392 6301338; Fax: +90 392 3651604}

%%%%%%%%%%%%%%%%%%%%%%%%%%%%%%%%%

 \section{Introduction}
 
Elementary complex analysis is a course that is taught for students of pure and applied mathematics, physics, and engineering programs. This course does not use the concept of Riemann surfaces and, therefore, presents logarithmic and power functions in a multivalued form. This makes elementary complex analysis insufficiently elementary. An implementation of Riemann surfaces, which treats the preceding and other multivalued functions as single-valued, comes later after getting knowledges in topology, manifolds, differential forms, integration on manifolds, etc. In most universities students of applied mathematics, physics, and engineering programs do not see this implementation. Is it possible to make the elementary part of complex analysis more elementary and at the same time be closer to the advanced part? 
 
 This paper demonstrates that if elementary complex analysis is interpreted in the frame of multiplicative calculus, then it turns to be more elementary. In fact, this is a consideration of complex analysis in one particular Riemann surface (generated by complex logarithm), but does not use heavy machinery of Riemann surfaces. Therefore, multiplicative version of complex analysis also becomes a good introductory basis for complex analysis on Riemann surfaces.
 
Elementary complex analysis starts with the field $\mathbb{C}$ of complex numbers and operations on them. Very soon by the de'Moivre's formula, which states that the number of $n$th roots of any nonzero complex number is exactly $n$, it becomes clear that the students should expect complications related to this formula. Later this result lies down to the definition of the complex power function that becomes multi-valued. Generally, it has countably many branches, it has finitely many branches just for rational powers that turns to a single branch for integer powers. Another elementary complex function, met at the beginning, is the argument function $\arg z$ with the principal branch $\mathrm{Arg}\,z$. This function has again countably many branches. The most important multivalued function is the complex logarithmic function $\log z=\ln |z|+i(\mathrm{Arg}\,z+2\pi n)$, where $|z|$ is the modulus of $z$ and $i$ is the imaginary unit.

Resolving the multivalued nature of these functions without using heavy machinery of Riemann surfaces can essentially simplify elementary complex analysis. To do so, a "trick" that considers only the principal branches is used. This idea is used in real analysis as well. For example, the familiar $\sin x$ of real variable has a multi-valued inverse. In order to make it single-valued, we just invert its restriction to the interval $[-\frac{\pi }{2},\frac{\pi }{2}]$. This suffices for the needs of real analysis. But for the elementary complex analysis, consideration of only principal branches is insufficient.

The multi-valued nature of the previously mentioned complex functions is due to their definition over the polar representation of complex numbers which is non-unique. In fact, there is a lack of space in the complex plane, that contains just one complex number, corresponding to its distinct polar representations. Introducing the Riemann surface of the complex logarithm solves this insufficiency. But the issue is how to do it for students who are not yet aware about topology, differential forms, integration on manifolds, etc. In this paper we are going to demonstrate that this issue can be overcame by using easily understandable concepts of non-Newtonian calculi, specifically, calculus generated by exponential function. 

The paper is constructed in the following way. In a few following sections we discuss some elements towards the basic ideas of the simplification and then in the concluding section of the paper we propose a draft outline of an alternative elementary complex analysis. Since the discussion is concentrated on the elementary part of complex analysis, we call it as complex calculus and the alternative elementary complex analysis as complex *calculus.

%%%%%%%%%%%%%%%%%%%%%%%%%%%%%%%%%%%%%%

\section{The system $\mathbb{B}$}

We differ distinct polar representations of the same complex number by introducing \begin{equation}
 \label{1}
 \mathbb{B}=\{ (r,\theta ):r>0,\ -\infty <\theta <\infty \} .
 \end{equation}
Zero will be considered out of $\mathbb{B}$ while it belongs to the system $\mathbb{C}$ of complex numbers and the system $\mathbb{R}$ of real numbers. This is because the addition operation of $\mathbb{C}$ as well as $\mathbb{R}$ is not fully functioning in $\mathbb{B}$ while the multiplication operation has a perfect extension. Therefore, it is reasonable to consider $\mathbb{B}$ as a multiplicative group rather than a field. One can see that in fact $\mathbb{B}$ is the Riemann surface of the complex logarithm although this could not be mentioned explicitly.

It is reasonable to call the subset 
 \[
 \mathbb{B}_\alpha =\{ (r,\theta ):r>0,\ \alpha -\pi \le \theta <\alpha +\pi \} 
 \] 
of $\mathbb{B}$ as an $\alpha $-branch of $\mathbb{B}$. In fact
 \[
 \mathbb{B}=\bigcup _{n=-\infty }^\infty \mathbb{B}_{2\pi n}\ \text{with}\ \mathbb{B}_{2\pi n}\cap \mathbb{B}_{2\pi m}=\varnothing \ \text{for}\ n\not= m .
 \]
We will use the identification
 \[
 \mathbb{B}_0 \ni (r,\theta )=re^{i\theta }\in \mathbb{C}\setminus \{0\},
 \] 
where $\theta $ stands for the principal argument of the complex number $z=re^{i\theta }$. In this regard, $\mathbb{C}\setminus \{ 0\} $ is a proper subset of $\mathbb{B}$. Similarly, nonzero real numbers belong to $\mathbb{B}$ with representations $x=(x,0)$ if $x>0$ and $x=(-x,-\pi )$ if $x<0$. 

To remove possible ambiguities, the elements of $\mathbb{B}$ will be denoted by boldface letters like $\mathbf{z}$. Every $\mathbf{z}=(r,\theta )\in \mathbb{B}$ produces the nonzero complex number $z=r(\cos \theta +i\sin \theta )$. This function will be denoted by $z=\mathbf{l}(\mathbf{z})$. In fact, the function $\mathbf{l}$ is the periodic extension from $\mathbb{B}_0$ to $\mathbb{B}$ of the preceding identification $\mathbb{B}_0=\mathbb{C}\setminus \{ 0\} $. In the sequel, we will use boldface letters like $\mathbf{f}$ for functions with domain or range essentially in $\mathbb{B}$. The symbols like $f$ will be used for functions with domain and range in $\mathbb{C}$. We will refer to an element from $\mathbb{R}$ as $r$-number, from $\mathbb{C}$ as $c$-number, and from $\mathbb{B}$ as $b$-number for brevity.  

Not everything is so excellent in $\mathbb{B}$. In particular, the vector addition and subtraction of $c$-numbers lose the sense for $b$-numbers. Instead, the multiplication and division operations are nicely  extendable to $\mathbb{B}$ as follows
 \begin{equation}
 \label{2}
 \mathbf{z}_1\mathbf{z}_2=(r_1,\theta _1)(r_2,\theta _2)=(r_1r_2,\theta _1+\theta _2)
 \end{equation}
and
 \begin{equation}
 \label{3}
\frac{\mathbf{z}_1}{\mathbf{z}_2}=\frac{(r_1,\theta _1)}{(r_2,\theta _2)}=(r_1/r_2,\theta _1-\theta _2).
 \end{equation}
Therefore, we should avoid addition and subtraction and study functions on $\mathbb{B}$ via differentiation and integration on the basis of multiplication and division operations. This can be achieved by implementation of multiplicative definitions of derivative and integral.

%%%%%%%%%%%%%%%%%%%%%%%%%%%%%%%%% 

 \section{Non-Newtonian Calculi}
 
Differential and integral calculus, which is briefly called Newtonian calculus, was created by Isaak Newton and Gottfried Wilhelm Leibnitz in the second half of the 17th century.  This calculus studies functions via differentiation and integration with reference to linear functions. This means that there is a system of relationships between certain mathematical concepts, which we name as CALCULUS (its complex part as COMPLEX CALCULUS), and Newtonian calculus is its realization in the form of comparison of all functions with linear functions. If we choose other reference functions, we could realize CALCULUS differently. This issue is similar to a (perfect) translation of a story from one language to another. In the second half of the 60th decade Michael Grossman and Robert Katz \cite{GK} pointed out the realizations of CALCULUS which are different from Newtonian one, calling them as non-Newtonian calculi.

Briefly, if $\alpha $ is a bijection from $\mathbb{R}$ to the interval $\mathbb{I}$, then addition and multiplication of $\mathbb{R}$ can be isometrically transferred to $\mathbb{I}$ by letting
 \[
 a\oplus _\alpha b=\alpha (\alpha ^{-1}(a)+\alpha ^{-1}(b))\ \ \text{and}\ \ a\otimes _\alpha b=\alpha (\alpha ^{-1}(a)\times\alpha ^{-1}(b)).
 \]
Similarly, for $\alpha $-difference and $\alpha $-ratio we let
 \[
 a\ominus _\alpha b=\alpha (\alpha ^{-1}(a)-\alpha ^{-1}(b))\ \ \text{and}\ \ a\oslash _\alpha b=\alpha (\alpha ^{-1}(a)\slash\alpha ^{-1}(b)).
 \]
These operations of $\alpha $-arithmetic simplify many formulae which are less usual in ordinary arithmetic. For example, in Einstein's theory of relativity the speeds $a$ and $b$ are added as
 \[
 \frac{a+b}{1+ab/c^{2}},
 \]
where $c$ is the speed of light. Let $\tilde{a}=a/c$ and $\tilde{b}=b/c$ be the respective relative speeds. Then (see \cite{B2})
 \[
 \frac{a+b}{1+ab/c^{2}}=\tilde{a}\oplus _{\mathrm{tanh}}\tilde{b},
 \]
where $\tanh x$ is the hyperbolic tangent function. This means that formulae in theory of relativity can be more visibly written in terms of $\tanh $-arithmetic.

Since the differentiation and integration are infinitesimal versions of subtraction and addition, for a suitable function $\alpha $, $\alpha $-arithmetic induces $\alpha $-derivative and $\alpha $-integral of functions on $\mathbb{I}$ as
 \[
 \frac{d_\alpha }{dx}f(x)=\alpha \bigg( \frac{d}{dx}\alpha ^{-1}(f(x)) \bigg)
 \]
and
 \[
 (\alpha )\int _a^bf(x)\,dx=\alpha \bigg( \int _a^b\alpha ^{-1}(f(x))\,dx\bigg) .
 \]
 
The most popular $\alpha $-calculus is exp-calculus with a reference to exponential function $\alpha (x)=e^x$. This calculus is also named as multiplicative calculus, or briefly *calculus, because exp-addition is the ordinary product:
 \[
 a\oplus _{\rm exp}b=e^{\ln a+\ln b}=ab.
 \]
*Calculus is a suitable calculus for growth related problems. This and other features of *calculus are discussed in \cite{B1, B2}. Numerical methods via *calculus are developed in \cite{BER, OB, ORBB, RA, ROK}.
% In \cite{B3} *calculus is used to derive a mathematical model of literary texts. In \cite{B6} multiplicative partial derivatives, line, and double integrals are handled. 
Since the addition on $\mathbb{B}$ has weaknesses while the multiplication is powerful, we can employ multiplicative type derivative and integral for functions $\mathbf{f}:B\subseteq \mathbb{B}\to \mathbb{B}$. 

The multiplicative type of concepts are briefly called as *concepts. Previously, an attempt to study *derivative and *integral have been made for functions $f:C\subseteq \mathbb{C}\to \mathbb{C}$ in \cite{B4, B5}. While *derivative was found satisfactory, *integral met difficulties such as it was defined as a multivalued function because complex log-function is multivalued. Therefore, to go on we need in logarithmic and exponential functions over the system $\mathbb{B}$, use them to create functions $\mathbf{f}:B\subseteq \mathbb{B}\to \mathbb{B}$, and establish *calculus for them.   

%%%%%%%%%%%%%%%%%%%%%%%%%%

\section{Logarithms and exponents over $\mathbb{B}$}

On the basis of $\log z=\ln |z|+i(\mathrm{Arg}\, z+2\pi n)$ for $z\in \mathbb{C}\setminus \{ 0\} $, extend log-function to $\mathbb{B}$ as  
 \begin{equation}
 \label{4}
 \mathbf{log}\,\mathbf{z}=\ln r+i\theta ,\ \mathbf{z}=(r,\theta )\in \mathbb{B},
 \end{equation}
where $\ln r$ is the natural logarithm of the $r$-number $r$. So, $\mathbf{log}:\mathbb{B}\to \mathbb{C}$ becomes a bijection. We define a new exponential function by   
 \begin{equation}
 \label{5}
 \mathbf{e}^z=(e^x,y),\ z=x+iy\in \mathbb{C},
 \end{equation}
as a function from $\mathbb{C}$ to $\mathbb{B}$. Here $e^a$ is a natural exponent of the $r$-number $a$ (the same symbol will be used for complex exponents in the form $e^z$ as well while $\mathbf{e}^z$ is the new exponential function with values in $\mathbb{B}$). It can verified that $\mathbf{log}\, \mathbf{e}^z=z$ and $\mathbf{e}^{\mathbf{log }\,\mathbf{z}}=\mathbf{z}$ for all $z\in \mathbb{C}$ and $\mathbf{z}\in \mathbb{B}$. So, they are the inverse of each other. These functions play an underlying role in setting calculus over $\mathbb{B}$. 

With the multiplication operation (\ref{2}) on $\mathbb{B}$ we can establish the following familiar properties of logarithm and exponent. For every $\mathbf{z}_1=(r_1,\theta _1)$ and $\mathbf{z}_2=(r_2,\theta _2)$ from $\mathbb{B}$, 
 \[
 \mathbf{log}\,(\mathbf{z}_1\mathbf{z}_2)=\mathbf{log}\,\mathbf{z}_1+\mathbf{log}\,\mathbf{z}_2
 \]
since
 \begin{align*}
 \mathbf{log}\,(\mathbf{z}_1\mathbf{z}_2)
  & =\mathbf{log}\,((r_1,\theta _1)(r_2,\theta _2))=\mathbf{log}\,(r_1r_2,\theta _1+\theta _2)=\ln (r_1r_2)+i(\theta _1+\theta _2)\\
  & =\ln r_1+\ln r_2+i\theta _1+i\theta _2=(\ln r_1+i\theta _1)+(\ln r_2+i\theta _2)\\
  & =\mathbf{log}\,(r_1,\theta _1)+\mathbf{log}\,(r_2,\theta _2)=\mathbf{log}\,\mathbf{z}_1+\mathbf{log}\,\mathbf{z}_2.
 \end{align*} 
Similarly, for every $z_1=x_1+iy_1$ and $z_2=x_2+iy_2$ from $\mathbb{C}$, 
 \[
 \mathbf{e}^{z_1+z_2}=\mathbf{e}^{z_1}\mathbf{e}^{z_2}
 \]
since 
 \begin{align*}
 \mathbf{e}^{z_1+z_2}
  & =e^{(x_1+x_2)+i(y_1+y_2)}=(e^{x_1+x_2},y_1+y_2)\\
  & =(e^{x_1}e^{x_2},y_1+y_2)=(e^{x_1},y_1)(e^{x_2},y_2)=\mathbf{e}^{z_1}\mathbf{e}^{z_2}. 
 \end{align*}
 
%%%%%%%%%%%%%%%%%%%%%%%%%%%%%%%

\section{Powers over $\mathbb{B}$} 

The multiplication operation on $\mathbb{B}$  induces a natural power of $\mathbf{z}=(r,\theta )$ as
 \[
 \mathbf{z}^n=(r,\theta )^n=(r^n, n\theta ),
 \]
and taking into consideration that the $r$-number $1$ has the form $(1,0)$ as a $b$-number, negative integer power is induced as
 \[
 \mathbf{z}^{-n}=(r,\theta )^{-n}=(1/r^n, -n\theta ),
 \] 
together with $\mathbf{z}^0=(1,0)\in \mathbb{B}$.  The $n$th root can be defined as a solution of $\mathbf{w}^n=\mathbf{z}$ as
 \[
 \mathbf{w}=\sqrt[n]{\mathbf{z}}=\sqrt[n]{(r,\theta )}=(\sqrt[n]{r}, \theta /n),
 \] 
and hence, the rational power as
 \[
 \mathbf{z}^p=(r,\theta )^p=(r^p, p\theta ).
 \]
The complex power $w=u+iv$ of the $b$-number $\mathbf{z}=(r,\theta )$ can be defined by 
 \begin{equation}
 \label{6}
 \mathbf{z}^w=\mathbf{e}^{w\log \mathbf{z}}=(e^{u\ln r-v\theta }, u\theta +v\ln r)
 \end{equation}
since
 \begin{align*}
 \mathbf{z}^w
  & =\mathbf{e}^{w\log \mathbf{z}}=\mathbf{e}^{(u+iv)(\ln r+i\theta )}\\
  & =\mathbf{e}^{(u\ln r-v\theta )+i(u\theta +v\ln r)}=(e^{u\ln r-v\theta }, u\theta +v\ln r)
 \end{align*}
and this agrees with the previously derived formula for rational powers. Thus the power function $\mathbf{f}(w)=\mathbf{z}^w$ for $w\in \mathbb{C}$ is a single-valued function. In particular, 
 \[
 \sqrt{(r,0)}=(\sqrt{r},0)\ \text{and}\  \sqrt{(r,-2\pi)}=(\sqrt{r},-\pi ).
 \]
This means that the square root of a positive $r$-number $r$ in the sense of $\mathbb{B}$ is unique and equals to $\sqrt{r}$. Its second root in the sense of $\mathbb{R}$, that is, $-\sqrt{r}$ is now the root of $(r,-2\pi )$ in the sense of $\mathbb{B}$. Therefore, the positive and negative roots of a positive $r$-number are the unique roots of distinct $b$-numbers which are identified in $\mathbb{R}$.  
  
%%%%%%%%%%%%%%%%%%%%%%%%%%%%%%%%% 

\section{Other functions}

Every single-valued function $f$ of complex calculus generates its analog in complex *calculus which we denote by $\mathbf{f}$. The link between $f$ and $\mathbf{f}$ is as follows:
 \[
 \mathbf{f}(\mathbf{z})=\mathbf{e}^{f(\mathbf{log}\,\mathbf{z})},\ \mathbf{z}\in B\subseteq \mathbb{B}.
 \] 
The return link is
 \[
 f(z)=\mathbf{log}\,(\mathbf{f}(\mathbf{e}^{z})),\ z\in C\subseteq \mathbb{C}.
 \] 
 
In fact, $\mathbf{e}^z$, $z\in \mathbb{C}\subseteq \mathbb{B}$, defined by (\ref{5}), is the analog of the function $e^z$, $z\in \mathbb{C}$. Noting that $z\in \mathbb{C}$ can be interpreted as $c$-number $z=x+iy=re^{i\theta }$ with $x=r\cos \theta $ and $y=r\sin \theta $ as well as $b$-number $z=(r,\theta )$, this can be deduces as follows: 
 \begin{align*}
 \mathbf{e}^{e^{\mathbf{log}\,z}}
  &=\mathbf{e}^{e^{\mathbf{log}\,(r,\theta )}} & (\textrm{writing}\ z\in \mathbb{C}\ \text{as}\ b\text{-number}\ (r,\theta ))  \\
  &=\mathbf{e}^{e^{\ln r+i\theta }} & (\textrm{definition\ of}\ \mathbf{log}\text{-function})\\
  &=\mathbf{e}^{r(\cos \theta +i\sin \theta )} & (\textrm{definition\ of\ exp-function})\\
  &=(e^{r\cos \theta },r\sin \theta ) & (\textrm{definition\ of}\ \mathbf{exp}\text{-function})\\
  &=(e^x,y) & (\textrm{using}\ x=r\cos \theta \ \text{and}\ y=r\sin \theta )\\ 
  &=\mathbf{e}^z & (\textrm{using}\ z=x+iy) 
 \end{align*}
In particular,
 \[
 \mathbf{e}^z=(e^{r\cos \theta },r\sin\theta )\ \text{for}\ z=re^\theta \in \mathbb{C}.
 \] 
 
This does not expand to multivalued functions such as logarithmic and power functions, which are already defined by (\ref{4}) and (\ref{6}), but allows to create analogs of trigonometric and hyperbolic functions. In such a way,
  \begin{align*}
 \mathbf{cosh}\,\mathbf{z}
  &=\mathbf{e}^{\frac{1}{2}(e^{\mathbf{log}\,\mathbf{z}}+e^{-\mathbf{log}\,\mathbf{z}})}\\
  &=\mathbf{e}^{\frac{1}{2}(e^{\ln r+i\theta }+e^{-\ln r-i\theta })}\\
  &=\mathbf{e}^{\frac{1}{2}(r(\cos \theta +i\sin \theta )+r^{-1}(\cos \theta -i\sin \theta ))}  \\
  &=\mathbf{e}^{\frac{r+r^{-1}}{2}\cos \theta }\mathbf{e}^{i{\frac{r-r^{-1}}{2}\sin \theta }}  \\
  &=\mathbf{e}^{\cosh (\ln r)\cos \theta }\mathbf{e}^{i\sinh (\ln r)\sin \theta }  \\
  &=(e^{\cosh (\ln r)\cos \theta },\sinh (\ln r)\sin \theta )  
 \end{align*}
Similarly, one can derive
 \begin{align*}
 \mathbf{sinh}\,\mathbf{z} & =(e^{\sinh (\ln r)\cos \theta },\cosh (\ln r)\sin \theta ),\\
 \mathbf{cos}\,\mathbf{z} & =(e^{\cosh \theta \cos (\ln r)},-\sinh \theta \sin (\ln r)),\\
 \mathbf{sin}\,\mathbf{z} & =(e^{\cosh \theta \sin (\ln r)},\sinh \theta \cos (\ln r)).
 \end{align*} 

%%%%%%%%%%%%%%%%%%%%%%%%%%%%%%%%%

\section{*Derivative of functions over $\mathbb{B}$}

According to Section 3, the *derivative of non-vanishing function $f:C\subset \mathbb{C}\to \mathbb{C}\setminus \{ 0\} $ is defined by
 \[
 f^*(z)=e^{(\log f(z))'}=e^{\frac{f'(z)}{f(z)}}
 \]
as a single-valued function \cite{B4}. This formula for a function $f:R\subseteq \mathbb{R}\to (0,\infty )$ looks like 
 \[
 f^*(x)=e^{(\ln f(x))'}=e^{\frac{f'(x)}{f(x)}}.
 \] 
%The partial *derivatives of $f:G\subseteq \mathbb{R}^2\to (0,\infty )$ are defined respectively in \cite{B6}. 
To extend *derivative to functions $\mathbf{f}:B\subseteq \mathbb{B}\to \mathbb{B}$, we will need the following.
 \begin{lemma}
 \label{L1}
Let $f$ be a non-vanishing function from some nonempty open connected subset $C$ of $\mathbb{C}$ to $\mathbb{C}$. Assume that $f$ has the rectangular and polar representations
 \[
 f(z)=u(r,\theta )+iv(r,\theta )=R(r,\theta )e^{i\Theta (r,\theta )}\ \text{for}\ z=re^{i\theta } .
 \]
If $f'(z)$ exists and, respectively, the Cauchy--Riemann conditions in polar form
 \[
 ru'_r=v'_\theta ,\ rv'_r=-u'_\theta
 \]
hold, then 
 \begin{description}
 \item [\rm (a)]$r(\ln R)'_r=\Theta '_\theta $ and $r\Theta '_r=-(\ln R)'_\theta $,
 \item [\rm (b)]$f'(z)=e^{i(\Theta -\theta )}(R'_r+iR\Theta '_r)$,
 \item [\rm (c)]$(\log f(z))'=((\ln R)'_r\cos \theta +\Theta '_r\sin \theta )+i(\Theta '_r\cos \theta -(\ln R)'_r\sin \theta ) $, assuming that $f$ transfers $C$ into a branch of log-function.
 \end{description}
 \end{lemma}
 \begin{proof}
In fact, part (a) expresses a version of the Cauchy--Riemann conditions when both arguments and values of a complex function are represented in the polar form. Most textbooks  do not include this version of the Cauchy--Riemann conditions. Therefore, we derive them because $b$-numbers are based on just polar representation. 

We start from $R=R(r,\theta )=\sqrt{u^2(r,\theta )+v^2(r,\theta )}$ and calculate  
 \[
 R'_r=\frac{uu'_r+vv'_r}{\sqrt{u^2+v^2}}=\frac{uu'_r+vv'_r}{R}
 \]
and
 \[
 R'_\theta =\frac{uu'_\theta +vv'_\theta }{\sqrt{u^2+v^2}}=\frac{uu'_\theta +vv'_\theta }{R}.
 \] 
Next, $\Theta (r,\theta )=\mathrm{atan2}\,(v(r,\theta ),u(r,\theta ))$, where the $\mathrm{atan2}$-function is the $\mathrm{arctan}$-function of two variables and
 \[
 \mathrm{atan2}\,'_u=\frac{-v}{u^2+v^2}\ \text{and}\ \mathrm{atan2}\,'_v=\frac{u}{u^2+v^2}.
 \]
Therefore,
 \[
 \Theta '_r=\frac{-vu'_r}{u^2+v^2}+\frac{uv'_r}{u^2+v^2}=\frac{uv'_r-vu'_r}{R^2}.
 \] 
Similarly,
 \[
 \Theta '_\theta =\frac{-vu'_\theta }{u^2+v^2}+\frac{uv'_\theta }{u^2+v^2}=\frac{uv'_\theta -vu'_\theta }{R^2}.
 \] 
Thus using Cauchy--Riemann conditions in polar form, we obtain
 \[
 rR'_r=\frac{r(uu'_r+vv'_r)}{R}=\frac{uv'_\theta -vu'_\theta }{R}=R\Theta '_\theta \ \Rightarrow \ r(\ln R)'_r=\Theta '_\theta 
 \]
and
 \[
 R'_\theta =\frac{uu'_\theta +vv'_\theta }{R}=\frac{-r(uv'_r -vu'_r)}{R} =-rR\Theta '_r \ \Rightarrow \ (\ln R)'_\theta =-r\Theta '_r ,
 \] 
proving part (a).

To prove part (b), we start from the system of equations
 \[
 \left\{ \begin{array}{l} R'_rR=uu'_r+vv'_r,\\
                                     \Theta '_rR^2=uv'_r-vu'_r. \end{array}\right. 
 \]  
In the matrix form
 \[
 \left[ \!\!\begin{array}{c} R'_rR \\ \Theta '_rR^2 \end{array}\!\!\right] =
 \left[ \!\!\begin{array}{cc} u & v \\ -v & u \end{array}\!\!\right]
 \left[ \!\!\begin{array}{c} u'_r \\ v'_r \end{array}\!\!\right] ,
 \] 
which implies
  \[
 \left[ \!\!\begin{array}{c} u'_r \\ v'_r \end{array}\!\!\right]=
 \frac{1}{R^2}\left[ \!\!\begin{array}{cc} u & -v \\ v & u \end{array}\!\!\right]
 \left[ \!\!\begin{array}{c} R'_rR \\ \Theta '_rR^2 \end{array}\!\!\right] .
 \]  
Therefore,
 \[
 u'_r=\frac{1}{R}uR'_r-v\Theta'_r,\ v'_r=\frac{1}{R}vR'_r+u\Theta '_r .
 \] 
Thus
 \begin{align*}
 f'(z) 
   & =e^{-i\theta }(u'_r+iv'_r)=e^{-i\theta }\bigg( \frac{R'_r(u+iv)}{R}+i\Theta '_r(u+iv)\bigg) \\
   & =e^{-i\theta }(u+iv)\bigg( \frac{R'_r}{R}+i\Theta '_r\bigg) =e^{-i\theta }Re^{i\Theta }\bigg( \frac{R'_r}{R}+i\Theta '_r\bigg) \\
   & =e^{i(\Theta -\theta )}(R'_r+iR\Theta '_r),
 \end{align*}  
proving part (b).

Finally,
 \begin{align*}
 (\log f(z))' 
    & =\frac{f'(z)}{f(z)}=\frac{e^{i(\Theta -\theta )}(R'_r+iR\Theta '_r)}{Re^{i\Theta }}\\
    & =e^{-i\theta }((\ln R)'_r+i\Theta '_r)=(\cos \theta -i\sin \theta )((\ln R)'_r+i\Theta '_r) \\
    & =((\ln R)'_r\cos \theta +\Theta '_r\sin \theta )+i(\Theta '_r\cos \theta -(\ln R)'_r\sin \theta ) ,
 \end{align*} 
proving part (c).
 \end{proof}

Now let $\mathbf{f}:B\subseteq \mathbb{B}\to \mathbb{B}$ be given. The argument of this function will be denoted by $(r,\theta )\in B\subseteq \mathbb{B}$ and the value by $(R,\Theta )\in \mathbb{B}$. Therefore, we can present this function as $\mathbf{f}(r,\theta )=(R(r,\theta ),\Theta (r,\theta ))$. For example, for the power function $\mathbf{f}=\mathbf{z}^w$, $\mathbf{z}\in \mathbb{B}$, where $w=u+iv\in \mathbb{C}$ is fixed, the functions $R$ and $\Theta $ were derived previously in the form
 \[
 R(r,\theta )=e^{u\ln r-v\theta }\ \text{and}\ \Theta (r,\theta )=u\theta +v\ln r.
 \] 
Motivated from *derivative in the real and complex cases, we will put the formula 
 \[
  \mathbf{f}^*(\mathbf{z})=\mathbf{e}^{(\mathbf{log}\,(\mathbf{f}(\mathbf{z}))'}
 \]
on the basis of *derivative of $\mathbf{f}$. Then by Lemma \ref{L1}(c) we can set the following.  
 \begin{definition}
 \label{D1}
{\rm A function $\mathbf{f}(r,\theta )=(R(r,\theta ),\Theta (r,\theta ))$ is said to be *differentiable at $(r,\theta )$ if $R$ and $\Theta $ have continuous partial derivatives and the Cauchy--Riemann conditions in (a) hold at $(r,\theta )$. If $\mathbf{f}$ is *differentiable, then its *derivative is defined by
 \[
 \mathbf{f}^*(\mathbf{z})=\big( e^{(\ln R)'_r\cos \theta +\Theta '_r\sin \theta },\Theta '_r\cos \theta -(\ln R)'_r\sin \theta \big) .
 \]
$\mathbf{f}$ is said to be *analytic on $B\subseteq \mathbb{B}$ if $\mathbf{f}$ is *differentiable at every $(r,\theta )\in B$.  }
 \end{definition}
 \begin{example}
 \label{E1}
{\rm The ordinary derivative of a constant function is equal to the neutral element of addition, that is zero. In real and complex *calculus the *derivative of a constant function is equal to the neutral element of multiplication, that is one. Let us determine whether the *derivative of $\mathbf{f}(\mathbf{z})=\mathbf{z}_0$ is equal to the neutral element $(1,0)$ of multiplication in $\mathbb{B}$ or not. Letting $\mathbf{z}=(r,\theta )$ and $\mathbf{z}_0=(r_0,\theta _0)$, we obtain $R(r,\theta )=r_0$ and $\Theta (r,\theta )=\theta _0$. The Cauchy--Riemann conditions (a) hold in the form
 \[
 r(\ln R)'_r=\Theta '_\theta =(\ln R)'_\theta =-r\Theta '_r=0. 
 \]
Therefore, $\mathbf{f}^*(z)$ exists. Then by Definition \ref{D1}, $\mathbf{f}^*(z)=(e^0,0)=(1,0)$.}
 \end{example} 
 \begin{example}
 \label{E2}
{\rm In real and complex *calculus the exponential function plays the role of the linear function and, therefore, $(e^z)^*=e$ (a constant that in real case shows that the value of this function instantaneously changes $e$ times while $(e^z)'=e^z$ shows that the value changes for $e^z$ units). For $\mathbf{e}^z$, we have
 \[
 R(r,\theta )=e^{r\cos \theta }\ \text{and}\ \Theta (r,\theta )=r\sin \theta .
 \]
The Cauchy--Riemann conditions hold in the form
 \[
r(\ln R)'_r=\Theta '_\theta =r\cos \theta     \ \text{and}\ (\ln R)'_\theta =-r\Theta '_r=-r\sin \theta . 
 \] 
We have
 \[
 (\ln R)'_r\cos \theta +\Theta '_r\sin \theta =\cos ^2\theta +\sin ^2\theta =1
 \]
and
 \[
 \Theta '_r\cos \theta -(\ln R)'_r\sin \theta =\sin \theta \cos \theta -\cos \theta \sin \theta =0,
 \]
implying $(\mathbf{e}^z)^*=(e,0)$.
}
 \end{example}
 \begin{example}
 \label{E3}
{\rm The *derivative of the identity function $f(z)=z$, $z\in \mathbb{C}\setminus \{ 0\} $, was calculated in \cite{B5} in the form $(z)^*=e^{1/z}$. Let us calculate *derivative of $\mathbf{f}(\mathbf{z})=\mathbf{z}$, $\mathbf{z}\in \mathbb{B}$. For this function, we have
 \[
 R(r,\theta )=r\ \text{and}\ \Theta (r,\theta )=\theta .
 \] 
Then 
 \[
r(\ln R)'_r=\Theta '_\theta =1 \ \text{and}\ (\ln R)'_\theta =-r\Theta '_r=0. 
 \] 
We have
 \[
(\ln R)'_r\cos \theta +\Theta '_r\sin \theta =\frac{\cos \theta }{r} 
 \] 
and 
 \[
  \Theta '_r\cos \theta -(\ln R)'_r\sin \theta=-\frac{\sin \theta }{r}.
 \] 
Therefore,
 \[
 (\mathbf{z})^*=
 \bigg( e^{\frac{1}{r}\cos \theta },-\frac{1}{r}\sin \theta \bigg) =\mathbf{e}^{\frac{1}{r}(\cos \theta -i\sin \theta )}=\mathbf{e}^{\frac{1}{r(\cos \theta +i\sin \theta )}}=\mathbf{e}^{\frac{1}{\mathbf{l}(\mathbf{z})}}.
 \]
One can see that $(\mathbf{z})^*$ is $\mathbb{C}$-valued if $-\pi \le -\frac{1}{r}\sin \theta <\pi $ or, equivalently, 
 \[
 -r\pi <\sin \theta \le r\pi .
 \]
}
 \end{example}  
 \begin{example}
 \label{E4}
{\rm The *derivative of $\log z$, $z\in \mathbb{C}\setminus \{ 0\} $, was calculated in \cite{B5} in the form $(\log z)^*=e^{1/z\log z}$ as a multi-valued function. Let us calculate *derivative of $\mathbf{log}\, \mathbf{z}$, $\mathbf{z}\in \mathbb{B}$. For this function, we have
 \[
 R(r,\theta )=\sqrt{\ln ^2r+\theta ^2}\ \text{and}\ \Theta (r,\theta )=\mathrm{atan2}\,(\theta ,\ln r).
 \] 
Then
 \[
r(\ln R)'_r=\Theta '_\theta =\frac{\ln r}{\ln ^2r+\theta ^2} \ \text{and}\ (\ln R)'_\theta =-r\Theta '_r=\frac{\theta }{\ln ^2r+\theta ^2}. 
 \] 
We have
 \[
(\ln R)'_r\cos \theta +\Theta '_r\sin \theta =\frac{\ln r\cos \theta -\theta \sin \theta }{r(\ln ^2r+\theta ^2)}
 \] 
and 
 \[
  \Theta '_r\cos \theta -(\ln R)'_r\sin \theta =-\frac{\theta \cos \theta +\ln r\sin \theta }{r(\ln ^2r+\theta ^2)}.
 \]
Therefore,
 \begin{align*}
 (\log \mathbf{z})^*
  & =\bigg( e^{\frac{\ln r\cos \theta -\theta \sin \theta }{r(\ln ^2r+\theta ^2)}},-\frac{\theta \cos \theta +\ln r\sin \theta }{r(\ln ^2r+\theta ^2)}\bigg) \\
  & \mathbf{e}^{\frac{\ln r\cos \theta -\theta \sin \theta -i(\ln r\sin \theta +\theta \cos \theta )}{r(\ln ^2r+\theta ^2)}}=\mathbf{e}^{\frac{\ln r-i\theta }{\ln ^2r+\theta ^2}\cdot \frac{\cos \theta -i\sin \theta }{r}} \\
  & =\mathbf{e}^{\frac{\overline{\mathbf{log }\,\mathbf{z}}}{|\mathbf{log }\,\mathbf{z}|}\cdot\frac{\overline{\mathbf{l}(\mathbf{z})}}{|\mathbf{l}(\mathbf{z})|}}=\mathbf{e}^{\frac{1}{\mathbf{l}(\mathbf{z})\mathbf{log }\,\mathbf{z}}}.
 \end{align*} 
By Cauchy--Schwarz inequality
 \[
 -\frac{1}{r\sqrt{\ln ^2r+\theta ^2}}\le -\frac{\theta \cos \theta +\ln r\sin \theta }{r(\ln ^2r+\theta ^2)}\le \frac{1}{r\sqrt{\ln ^2r+\theta ^2}}
 \] 
Therefore, for $r\pi \sqrt{\ln ^2r+\theta ^2}>1$, $(\log \mathbf{z})^*$ is $\mathbb{C}$-valued. }
 \end{example} 
 \begin{example}
 \label{E5}
{\rm The *derivative of the real power function $x^a$, $x>0$, where $a\in \mathbb{R}$ is fixed, equals to $(x^a)^*=e^{\frac{a}{x}}$. Let us calculate *derivative of $\mathbf{z}^w$, $\mathbf{z}\in \mathbb{B}$, assuming that $w=u+iv\in \mathbb{C}$ is fixed. For this function, we have
 \[
 R(r,\theta )=e^{u\ln r-v\theta }\ \text{and}\ \Theta (r,\theta )=u\theta +v\ln r.
 \] 
The Cauchy--Riemann conditions hold in the form 
 \[
r(\ln R)'_r=\Theta '_\theta =u \ \text{and}\ (\ln R)'_\theta =-r\Theta '_r=-v. 
 \] 
We have
 \[
(\ln R)'_r\cos \theta +\Theta '_r\sin \theta =\frac{u\cos \theta +v\sin \theta }{r}
 \] 
and 
 \[
  \Theta '_r\cos \theta -(\ln R)'_r\sin \theta =\frac{v\cos \theta -u\sin \theta }{r}.
 \] 
Therefore,
 \begin{align*}
 (\mathbf{z}^w)^*
  & =(e^{(u\cos \theta +v\sin \theta )/r},(v\cos \theta -u\sin \theta )/r)\\
  & =\mathbf{e}^{\frac{1}{r}(u\cos \theta +v\sin \theta )+\frac{i}{r}(v\cos \theta -u\sin \theta )}.
 \end{align*}
Since
 \begin{align*}
 \frac{w}{\mathbf{l}(\mathbf{z})}
  & =\frac{u+iv}{r(\cos \theta +i\sin\theta )}=\frac{(u+iv)(\cos \theta -i\sin\theta )}{r}\\
  & =\frac{u\cos \theta +v\sin \theta }{r}+i\frac{v\cos \theta -u\sin \theta }{r},
 \end{align*}
we obtain
 \[
 (\mathbf{z}^w)^*=\mathbf{e}^\frac{w}{\mathbf{l}(\mathbf{z})}.
 \] 
By Cauchy--Schwarz inequality,
 \[
 -|w|=-\sqrt{u^2+v^2}\le v\cos \theta -u\sin \theta \le \sqrt{u^2+v^2}=|w|
 \]
Therefore, for $|w|<r\pi $, $(\mathbf{z}^w)^*$ is $\mathbb{C}$-valued. }
 \end{example}          
The following properties of *derivative can be easily verified:
 \begin{description}
 \item [\rm (i)] $(\mathbf{z}_0\mathbf{f})^*(\mathbf{z})=\mathbf{f}^*(\mathbf{z})$ for constant $\mathbf{z}_0\in \mathbb{B}$.
 \item [\rm (ii)] $(\mathbf{f}\mathbf{g})^*(\mathbf{z})=\mathbf{f}^*(\mathbf{z})\mathbf{g}^*(\mathbf{z})$.
 \item [\rm (iii)] $(\mathbf{f}/\mathbf{g})^*(\mathbf{z})=\mathbf{f}^*(\mathbf{z})/\mathbf{g}^*(\mathbf{z})$.
 \end{description}
For example, the proof of (ii) is as follows. Letting 
 \[
 \mathbf{z}=(r,\theta ),\ \mathbf{f}(\mathbf{z})=(R(r,\theta ),\Theta (r,\theta )),\ \mathbf{g}(\mathbf{z})=(\tilde{R}(r,\theta ),\tilde{\Theta }(r,\theta )),
 \]
we have 
 \[
 (\mathbf{f}\mathbf{g})(\mathbf{z})=(R(r,\theta )\tilde{R}(r,\theta ),\Theta (r,\theta )+\tilde{\Theta }(r,\theta )).
 \]
Therefore, 
 \begin{align*}
 (\mathbf{f}\mathbf{g})^*(\mathbf{z})=
  & \big( e^{(\ln (R\tilde{R}))'_r\cos \theta +(\Theta +\tilde{\Theta })'_r\sin \theta },((\Theta +\tilde{\Theta })'_r\cos \theta -\ln (R\tilde{R}))'_r\sin \theta \big) \\
  =& \big( e^{(\ln R)'_r\cos \theta +\Theta '_r\sin \theta },\Theta '_r\cos \theta -(\ln R)'_r\sin \theta \big) \\
  & \times \big( e^{(\ln \tilde{R})'_r\cos \theta +\tilde{\Theta }'_r\sin \theta },\tilde{\Theta }'_r\cos \theta -(\ln \tilde{R})'_r\sin \theta \big) \\
  = & \mathbf{f}^*(\mathbf{z})\mathbf{g}^*(\mathbf{z}).
 \end{align*} 
The other properties can be proved similarly. Moreover, remaining theorems of complex differentiation can be stated and proved in terms of *derivative. More importantly, we pass to *integral in the next section.

%%%%%%%%%%%%%%%%%%%%%%%%%%%%%%%%%

\section{*Integral of functions over $\mathbb{B}$}

\begin{definition}
\label{D2}
{\rm Given $\mathbf{f}:B\subseteq \mathbb{B}\to \mathbb{B}$ by $\mathbf{f}(\mathbf{z})=(R(r,\theta ),\Theta (r,\theta ))$ for $\mathbf{z}=(r,\theta )$, we assume that $C$ is a contour in $B$ with a parameterization $\mathbf{z}(t)=(r(t),\theta (t))$, $a\le t\le b$. Then we define the *integral of $\mathbf{f}$ along the curve $C$ by
 \begin{equation}
 \label{7}
 \int _C\mathbf{f}(\mathbf{z})^{d\mathbf{z}}=\mathbf{e}^{\int _CP\,dr-Q\,d\theta +i\int _CM\,dr+N\,d\theta }.
 \end{equation}
assuming that the ordinary line integrals in the right side exist and
 \begin{align*}
 P(r,\theta ) &=\ln R\cos\theta -\Theta \sin \theta ,\\
 Q(r,\theta ) &=r\ln R\sin \theta +r\Theta \cos \theta ,\\
 M(r,\theta ) &=\ln R\sin\theta +\Theta \cos \theta ,\\
 N(r,\theta ) &=r\ln R\cos \theta -r\Theta \sin \theta .
 \end{align*}
}
 \end{definition}  
  
This integral has the following obvious properties:
 \begin{description}
 \item [\rm (i)] $\int _C\mathbf{f}(\mathbf{z})^{d\mathbf{z}}=\int _{C_1}\mathbf{f}(\mathbf{z})^{d\mathbf{z}}\int _{C_2}\mathbf{f}(\mathbf{z})^{d\mathbf{z}}$, where $C=C_1+C_2$.
 \item [\rm (ii)] $\int _C\mathbf{f}(\mathbf{z})\mathbf{g}(\mathbf{z})^{d\mathbf{z}}=\int _C\mathbf{f}(\mathbf{z})^{d\mathbf{z}}\int _C\mathbf{g}(\mathbf{z})^{d\mathbf{z}} $.
 \item [\rm (iii)] $\int _C\mathbf{f}(\mathbf{z})/\mathbf{g}(\mathbf{z})^{d\mathbf{z}}=\int _C\mathbf{f}(\mathbf{z})^{d\mathbf{z}}\big/ \int _C\mathbf{g}(\mathbf{z})^{d\mathbf{z}} $.
  \end{description} 
For example, property (ii) can be proved as follows. Let $P$, $Q$, $M$, and $N$ be the preceding functions associated with $\mathbf{f}$ and denote the respective functions associated with $\mathbf{g}$ by $\tilde{P}$, $\tilde{Q}$, $\tilde{M}$, and $\tilde{N}$. Then the respective functions associated with $\mathbf{f}\mathbf{g}$ are
 \begin{align*}
 \mathcal{P}(r,\theta ) &=\ln (R\tilde{R})\cos\theta -(\Theta +\tilde{\Theta })\sin \theta ,\\
 \mathcal{Q}(r,\theta ) &=r\ln (R\tilde{R})\sin \theta +r(\Theta +\tilde{\Theta })\cos \theta ,\\
 \mathcal{M}(r,\theta ) &=\ln (R\tilde{R})\sin\theta +(\Theta +\tilde{\Theta })\cos \theta ,\\
 \mathcal{N}(r,\theta ) &=r\ln (R\tilde{R})\cos \theta -r(\Theta +\tilde{\Theta })\sin \theta .
 \end{align*} 
Therefore,  
 \begin{align*}
 \int _C\mathbf{f}(\mathbf{z})\mathbf{g}(\mathbf{z})^{d\mathbf{z}}
  & =\mathbf{e}^{\int _C\mathcal{P}\,dr-\mathcal{Q}\,d\theta +i\int _C\mathcal{M}\,dr+\mathcal{N}\,d\theta } \\
  & =\mathbf{e}^{\int _CPdr-Q\,d\theta +i\int _CM\,dr+N\,d\theta }\mathbf{e}^{\int _C\tilde{P}\,dr-\tilde{Q}\,d\theta +i\int _C\tilde{M}\,dr+\tilde{N}\,d\theta }\\
  & =\int _C\mathbf{f}(\mathbf{z})^{d\mathbf{z}}\int _C\mathbf{g}(\mathbf{z})^{d\mathbf{z}}.
  \end{align*}
  
We are mainly interested in whether *integral inverts *differentiation or not.   
 \begin{theorem} [Fundamental theorem of *calculus]
 \label{T1} 
Let $\mathbf{f}$ be *analytic on a connected set $B\subseteq \mathbb{B}$ and let $C$ be a contour in $B$ with the initial and end points $\mathbf{z}_1$ and $\mathbf{z}_2$, respectively. Then
  \[
  \int _C\mathbf{f}^*(\mathbf{z})^{d\mathbf{z}}=\frac{\mathbf{f}(\mathbf{z}_2)}{\mathbf{f}(\mathbf{z}_1)}.
  \] 
 \end{theorem}
 \begin{proof}
Let $C$ be parameterized by $\mathbf{z}(t)=(r(t),\theta (t))$, $a\le t\le b$, and let $\mathbf{f}(\mathbf{z})=(R(r,\theta ),\Theta (r,\theta ))$. Then
 \[
 \mathbf{f}^*(\mathbf{z})=\mathbf{e}^{(\ln R)'_r\cos \theta +\Theta '_r\sin \theta +i(\Theta '_r\cos \theta -(\ln R)'_r\sin \theta )} .
 \] 
We can calculate the functions $P$, $Q$, $M$, and $N$, associated with $\mathbf{f}$, in the form: 
 \begin{align*}
 P & =((\ln R)'_r\cos \theta +\Theta '_r\sin \theta )\cos \theta -(\Theta '_r\cos \theta -(\ln R)'_r\sin \theta )\sin \theta =(\ln R)'_r,\\ 
 Q & =r((\ln R)'_r\cos \theta +\Theta '_r\sin \theta )\sin \theta +r(\Theta '_r\cos \theta -(\ln R)'_r\sin \theta )\cos \theta =r\Theta '_r,\\
 M & =((\ln R)'_r\cos \theta +\Theta '_r\sin \theta )\sin \theta +(\Theta '_r\cos \theta -(\ln R)'_r\sin \theta )\cos \theta =\Theta '_r,\\
 N & =r((\ln R)'_r\cos \theta +\Theta '_r\sin \theta )\cos \theta -r(\Theta '_r\cos \theta -(\ln R)'_r\sin \theta )\sin \theta =r(\ln R)'_r.        
 \end{align*}
Therefore,
 \[
\int _C\mathbf{f}^*(\mathbf{z})^{d\mathbf{z}}=\mathbf{e}^{\int _C(\ln R)'_r\,dr-r\Theta '_r\,d\theta +i\int _C\Theta '_r\,dr+r(\ln R)'_r\,d\theta }.
 \]
By Cauchy--Riemann conditions,
 \begin{align*}
\int _C\mathbf{f}^*(\mathbf{z})^{d\mathbf{z}}
 & =\mathbf{e}^{\int _C(\ln R)'_r\,dr+(\ln R)'_\theta \,d\theta +i\int _C\Theta '_r\,dr+\Theta '_\theta \,d\theta } \\
 & =\mathbf{e}^{\ln R(r(b),\theta (b))-\ln R(r(a),\theta (a))+i(\Theta (r(b),\theta (b))-\Theta (r(a),\theta (a)))}\\  
 & =\frac{\mathbf{e}^{\ln R(r(b),\theta (b))+i\Theta (r(b),\theta (b)))}}{\mathbf{e}^{\ln R(r(a),\theta (a))+i\Theta (r(a),\theta (a)))}}\\
 & =\frac{(R(r(b),\theta (b)),\Theta (r(b),\theta (b))}{(R(r(a),\theta (a)),\Theta (r(a),\theta (a))}=\frac{\mathbf{f}(\mathbf{z}_2)}{\mathbf{f}(\mathbf{z}_1)},
 \end{align*}   
proving the theorem.
 \end{proof}
 \begin{remark}
 \label{R1}
{\rm Let $C_1$ and $C_2$ be two contours in the connected set $B\subseteq \mathbb{B}$ with the same initial and end points $\mathbf{z}_1$ and $\mathbf{z}_2$, respectively, and let $\mathbf{f}$ be *analytic on $B$. Then by Theorem \ref{T1},
 \[
 \int _{C_1}\mathbf{f}(\mathbf{z})^{d\mathbf{z}}=\int _{C_2}\mathbf{f}(\mathbf{z})^{d\mathbf{z}},
 \]
that means the *integral is independent on the shape of the contours. Therefore, this integral can be denoted as 
 \[
 \int _{\mathbf{z}_1}^{\mathbf{z}_2}\mathbf{f}(\mathbf{z})^{d\mathbf{z}}.
 \] 
 }
 \end{remark}
 \begin{theorem}
 \label{T2}
Let $\mathbf{f}$ be *analytic on a connected set $B\subseteq \mathbb{B}$ and let $C$ be a closed contour in $B$. Then
  \[
  \int _C\mathbf{f}^*(\mathbf{z})^{d\mathbf{z}}=(1,0)=\mathbf{e}^0.
  \]  
 \end{theorem}
 \begin{proof}
Under notation in Definition \ref{D2}, we have 
 \[
 Q'_r+P'_\theta =\ln R\sin \theta +\Theta \cos \theta -\ln R\sin \theta -\Theta \cos \theta =0
 \]
and
 \[
 N'_r-M'_\theta =\ln R\cos \theta +\Theta \sin \theta -\ln R\cos \theta -\Theta \sin \theta =0.
 \] 
Therefore, applying Green's theorem to (\ref{7}), we obtain the conclusion of the theorem. 
 \end{proof}
 \begin{example}
 \label{E6}
{\rm By Example \ref{E3}, $(\mathbf{z})^*=\mathbf{e}^{\frac{1}{\mathbf{l}(\mathbf{z})}}$. We are seeking to find the analog of this formula in complex calculus. In the real case, this formula states
 \[
 e^{\frac{1}{x}}=f^*(x)=e^{\frac{f'(x)}{f(x)}}\ \Rightarrow \ (\ln f(x))'=\frac{1}{x}\ \Rightarrow \ f(x)=x+c.
 \]   
Therefore, the equality $(\mathbf{z})^*=\mathbf{e}^{\frac{1}{\mathbf{l}(\mathbf{z})}}$ is a complex *version of $(x)'=1$. Then what is the analog of 
 \[
 \oint _{|z|=a}z\,dz=0
 \]
in complex *calculus? The analog of the counterclockwise oriented circle $|z|=a$ in $\mathbb{C}$ is the contour $C$ in $\mathbb{B}$ with the parameterization $(r,\theta )=(a, t)$, $-\pi \le t\le \pi $, which is no longer closed in $\mathbb{B}$. It has distinct initial and end points $(a,-\pi )$ and $(a,\pi )$, respectively, in disjoint branches. On the basis of Theorem \ref{T1} and Example \ref{E3}, we can calculate
 \[
 \int _C \big( \mathbf{e}^{\frac{1}{\mathbf{l}(\mathbf{z})}}\big)^{d\mathbf{z}}=\frac{\mathbf{z}_2}{\mathbf{z}_1}=\frac{(a,\pi )}{(a,-\pi )}=(1,2\pi )=\mathbf{e}^{2\pi i}.
 \]
While expecting to obtain the $b$-number $(1,0)$ (e.g. $c$-number 1 that is a neutral element of multiplication), we obtained $(1,2\pi )$. This is because in $\mathbb{C}$ the $b$-numbers $(1,0)$ and $(1,2\pi )$ are identified.} 
 \end{example} 
 \begin{example}
 \label{E7}
{\rm By Example \ref{E4}, $(\mathbf{log}\, \mathbf{z})^*=\mathbf{e}^{\frac{1}{\mathbf{l}(\mathbf{z})\mathbf{log}\, \mathbf{z}}}$. In the real case, this formula states
 \[
 e^{\frac{1}{x\ln x}}=f^*(x)=e^{\frac{f'(x)}{f(x)}}\ \Rightarrow \ (\ln f(x))'=\frac{1}{x\ln x}\ \Rightarrow \ f(x)=\ln x+c.
 \]   
Therefore, the equality $(\mathbf{log }\,\mathbf{z})^*=\mathbf{e}^{\frac{1}{\mathbf{l}(\mathbf{z})\mathbf{log}\, \mathbf{z}}}$ is a complex *version of $(\ln x)'=\frac{1}{x}$. Based on this we would like to find *version of the important formula
 \[
 \oint _{|z|=a}\frac{dz}{z}=2\pi i\ \text{for}\ a>0. 
 \]
Similar to Example \ref{E6}, we look to the contour $C$ in $\mathbb{B}$ with parameterization $(r,\theta )=(a, t)$, $-\pi \le t\le \pi $. On the basis of Theorem \ref{T1} and Example \ref{E4}, we have
 \[
 \int _C\big( \mathbf{e}^{\frac{1}{\mathbf{l}(\mathbf{z})\mathbf{log}\, \mathbf{z}}}\big) ^{d\mathbf{z}}=\frac{\mathbf{log}\, (a,\pi )}{\mathbf{log}\, (a,-\pi )}=e^{2i\arctan \frac{\pi }{\ln a}}. 
 \] 
The dependence on $a$ of the calculated result looks like strange, but this is accumulated from *nature of calculations. Indeed, converting the result from exp-arithmetic to ordinary arithmetic by formula $z_1-z_2=\mathbf{log}\, (\mathbf{e}^{z_1}\ominus _\mathrm{exp}\mathbf{e}^{z_2})$, we obtain exactly $2\pi i$:
 \begin{align*}
 \mathbf{log}\, \big( \mathbf{e}^{\mathbf{log}\, (a,\pi )}\ominus _\mathrm{exp}\mathbf{e}^{\mathbf{log}\, (a,-\pi )}\big) 
  & =\mathbf{log}\, \mathbf{e}^{\ln a+i\pi -\ln a+i\pi }\\
  & =\mathbf{log }\,(e^0,2\pi )=\ln 1+2\pi i=2\pi i.
 \end{align*}
This example demonstrates that the *integral responds to residues in a specific form. More precisely, the residues appear as a result of calculation of *integrals over non-closed contours by using the fundamental theorem of *calculus. In other words, the residue theory, which is a some sort of attachment to fundamental theorem of complex calculus and covers the cases out of this theorem, now becomes an integral part of fundamental theorem of complex *calculus.  }
 \end{example} 
 
%%%%%%%%%%%%%%%%%%%%%%%%%%%%%%%%%

\section{Conclusion}

On the basis of the preceding discussion we find it reasonable to develop other topics of elementary complex calculus in multiplicative frame in details with the aim to propose an alternative course on elementary complex calculus. The draft outline of this course may be seen as follows:
 \begin{description}
 \item [$\bullet $] Definition of $\mathbb{C}$ and algebra on $\mathbb{C}$ including the de'Moivre's formula. The Euler's formula, periodic exponential function, multivalued logarithmic and power functions. Complications related to multivalued functions.
 \item [$\bullet $] Definition of $\mathbb{B}$ without mentioning that it is the Riemann surface of log-function. Algebra on $\mathbb{B}$ emphasizing on global nature of multiplication on $\mathbb{B}$.
 \item [$\bullet $] Review on non-Newtonian calculi in the real case emphasizing on *calculus. Pointing out that *calculus is a presentation of CALCULUS and it is equivalent to Newtonian presentation. 
  \item [$\bullet $] Definitions of elementary functions (logarithmic, exponential, power, trigonometric, hyperbolic) over $\mathbb{B}$. 
  \item [$\bullet $] Limit and continuity of $\mathbf{f}(\mathbf{z})=(R(r,\theta ),\Theta (r,\theta ))$. Although this part was not concerned in the preceding discussion, it can be developed on the basis of limit and continuity of functions $R$ and $\Theta $ of two real variables. It is suitable to discuss limits at zero and infinity. 
 \item [$\bullet $] *Differentiation with Cauchy--Riemann conditions from Lemma \ref{L1}(a). The *derivative of $\mathbf{f}(\mathbf{z})=(R(r,\theta ),\Theta (r,\theta ))$ can be developed by using *derivatives of real-valued function $R$ and $\Theta $ through limit and then Definition \ref{D1} can be obtained.
 \item [$\bullet $] *Integration  with fundamental theorem of *calculus. Again *integrals can be developed by using of integral products (instead of integral sums) and then Definition \ref{D2} can be obtained. 
 \item [$\bullet $] No need for residues and singularities in complex *calculus since complex *integrals  can be calculated just by fundamental theorem of *calculus. Also, the functions of complex *calculus have no singularities leading to residues. Removal of this part of complex calculus would provide a great economy of time which can be used for next items.
 \item [$\bullet $] Conformal mappings. This can be developed in *context in more details because of saving time. The preceding discussion does not include this topic which needs additional investigation.
 \item [$\bullet $] In addition, working on new topics in *context is recommended. For example, complex Fourier series in *context seems to be appropriate. The preceding discussion does not include this topic. It is notable that Fourier series fit to Newtonian calculus inappropriately in comparison to Taylor series because the reference function of Newtonian calculus is linear. Fourier series do not fit to real *calculus as well because sine and cosine functions have zeros. We think that an appropriate calculus for Fourier series is complex *calculus over $\mathbb{B}$ because its reference function is non-vanishing exponential function $\mathbf{e}^z$ and also complex Fourier series have exponential form. This issue needs additional investigation.     
 \end{description}

Of course, there are a lot of problems related to interpretation of remaining theorems of complex analysis (elementary or not) in *form. Besides conformal mappings and Fourier series, mentioned previously, it is challenging a study of the the fundamental theorem of algebra in *context. Noting that in *calculus polynomials become products of power functions and motivated from the single-valued nature of power functions in complex *calculus, fundamental *theorem of algebra is expecting to be seen as an existence of a unique solution. 

This paper just demonstrates that the Riemann surface of complex logarithm can be replaced by methods of multiplicative calculus. This raises the following challenging question: Can complex analysis on Riemann surfaces be easily covered by consideration different non-Newtonian calculi?  This requires a wide research of interested experts in the field. We expect many items of complex analysis to be changed, mainly being simplified.
 
%%%%%%%%%%%%%%%%%%%%%%%%%%%%%%%%%%% 


\begin{thebibliography}{1}

\bibitem{A}
Ahlfors, L. V.: Complex Analysis, 3rd ed. McGraw-Hill, New York (1979)

\bibitem{B1}
Bashirov, A. E., Kurp\i{}nar, E., \"{O}zyap\i{}c\i{}, A.: Multiplicative calculus and its applications.
Journal of Mathematical Analysis and Applications. (1) {\bf 337}, 36--48 (2008)

\bibitem{B2}
Bashirov, A. E., M\i{}s\i{}rl\i{}, E., Tando\u{g}du, Y., \"{O}zyap\i{}c\i{}, A.: On modeling with multiplicative differential equations. Applied Mathematics -- Journal of Chinese Universities. (4) {\bf 26}, 425--438 (2011)

%\bibitem{B3}
%Bashirov, A. E., Bashirova, G.: Literature as a diffusion process. In: Brownian Motion: Elements, Dynamics and Applications, Editors: M.A.McKibben and M.Webster. Nova Science Publications, New York, 201-218 (2015)

%\bibitem{B6}
%Bashirov, A. E., On line and double multiplicative integrals. TWMS Journal on Applied and Engineering Mathematics. (1) {\bf 3}, 103--107 (2013)

\bibitem{B4}
Bashirov, A. E., Riza, M.: On complex multiplicative differentiation. TWMS Journal on Applied and Engineering Mathematics. (1) {\bf 1}, 75--85 (2011)

\bibitem{B5}
Bashirov, A. E., Norozpour, S.: On complex multiplicative integration. TWMS Journal on Applied and Engineering Mathematics (accepted).

\bibitem{BER}
Bilgehan, B., Emina{\u{g}}a, B., Riza, M.: New Solution Method for Electrical Systems Represented by Ordinary Differential Equation. Journal of Circuits, Systems and Computers. (2) {\bf 25}, doi: 10.1142/S0218126616500110 (2016)

\bibitem{GK}
Grossman, M., Katz, R.: Non-Newtonian Calculi. Lee Press, Pigeon Cove, MA (1972)

 \bibitem {OB} 
\"{O}zyap\i{}c\i{}, A., Bilgehan, B.: Finite product representation via multiplicative calculus and its applications to exponential signal processing. Numerical Algorithms. (2) {\bf 71}, 475--489 (2016)

\bibitem {ORBB} 
\"{O}zyap\i{}c\i{}, A., Riza, M., Bilgehan, B., Bashirov, A. E.: On multiplicative and Volterra minimization method. Numerical Algorithms. (3) {\bf 67}, 623--636 (2014)

\bibitem {RA}
Riza, M., Akt{\"o}re, H.: The Runge--Kutta method in geometric multiplicative calculus. LMS Journal of Computation and Mathematics. (1) {\bf 18}, 539--554 (2015)

\bibitem{ROK}
Riza, M., \"{O}zyap\i{}c\i{}, A., Kurp\i{}nar, E., Multiplicative finite difference methods. Quarterly of Applied Mathematics. (4) {\bf 67}, 745--754 (2009)

\bibitem{BC}
Brown, J. W., Churchill, R. V.: Complex variables and applications, 8th ed. McGraw-Hill, New York (2009)

\bibitem{S}
Sarason, D.: Complex Functions Theory. Amer. Math Society, Providence, RI (2007)

 \end{thebibliography}
 \end{document}